\documentclass{amsart}[12pt]

\usepackage{amscd,amssymb}
\usepackage{comment}

\usepackage{color}

\newtheorem{theorem}{Theorem}[section]
\newtheorem{lemma}[theorem]{Lemma}
\newtheorem{proposition}[theorem]{Proposition}
\newtheorem{corollary}[theorem]{Corollary}

\theoremstyle{definition}

\newtheorem{remark}[theorem]{Remark}

\newcommand{\GL} {\mathrm{GL}}
\newcommand{\SL} {\mathrm{SL}}
\newcommand{\OO} {\mathrm{O}}
\newcommand{\SO} {\mathrm{SO}}
\newcommand{\Sp} {\mathrm{Sp}}

\def\CC {{\mathbb C}}     %% complex numbers
\def\ZZ {{\mathbb Z}}     %% integers
\begin{document}

\title[On the $\GL(n)$-module structure of relatively free algebras]
{On the $\GL(n)$-module structure of Lie nilpotent associative relatively free algebras}

\author[Elitza Hristova]
{Elitza Hristova}
\address{Institute of Mathematics and Informatics,
Bulgarian Academy of Sciences,
Acad. G. Bonchev Str., Block 8,
1113 Sofia, Bulgaria}
\email{e.hristova@math.bas.bg}

\thanks{Partially supported by Grant KP-06 N32/1 of the Bulgarian National Science Fund.}

\subjclass[2010]{13A50; 16R10; 20G05}

\keywords{Lie nilpotent associative relatively free algebras, product of commutators, $\GL(n)$-module structure, invariant theory, Hilbert series.}

\begin{abstract}
Let $K\left\langle X \right\rangle$ denote the free associative algebra generated by a set $X = \{x_1, \dots, x_n\}$ over a field $K$ of characteristic $0$.
Let $I_p$, for $p \geq 2$, denote the two-sided ideal in $K\left\langle X \right\rangle$ generated by all commutators of the form $[u_1, \dots, u_p]$,
where $u_1, \dots, u_p \in K\left\langle X \right\rangle$. We discuss the $\GL(n, K)$-module structure of the quotient $K\left\langle X \right\rangle / I_{p+1}$ for all $p \geq 1$ under the standard diagonal action.
We give a bound on the values of partitions $\lambda$ such that the irreducible $\GL(n, K)$-module $V_{\lambda}$ appears
in the decomposition of $K\left\langle X \right\rangle / I_{p+1}$ as a $\GL(n, K)$-module. As an application,
we take $K = \CC$ and we consider the algebra of invariants $(\CC\left\langle X \right\rangle / I_{p+1})^G$
for $G = \SL(n, \CC)$, $\OO(n, \CC)$, $\SO(n, \CC)$, or $\Sp(2s, \CC)$ (for $n=2s$).
By a theorem of Domokos and Drensky, $(\CC\left\langle X \right\rangle / I_{p+1})^G$ is finitely generated.
We give an upper bound on the degree of generators of $(\CC\left\langle X \right\rangle / I_{p+1})^G$ in a minimal generating set.
In a similar way, we consider also the algebra of invariants $(\CC\left\langle X \right\rangle / I_{p+1})^{G}$, where $G=\mathrm{UT}(n, \CC)$,
and give an upper bound on the degree of generators in a minimal generating set.
These results provide useful information about the invariants in $\CC\left\langle X \right\rangle^G$ from the point of view of Classical Invariant Theory.
In particular, for all $G$ as above we give a criterion when a $G$-invariant of $\CC\left\langle X \right\rangle$ belongs to $I_p$.
\end{abstract}

\maketitle
\section{Introduction}
Let $V$ be a finite dimensional vector space over a field $K$ of characteristic zero and let $X = \{x_1, \dots, x_n\}$ be a basis for $V$.
We consider the free associative algebra $K\left\langle X \right\rangle$ generated by $X$, which can be naturally identified with the tensor algebra $T(V)$ of $V$.
We define inductively left-normed long commutators in $K\left\langle X \right\rangle$ in the usual way: $[u_1, u_2] = u_1u_2- u_2u_1$
and for $i \geq 3$, we have $[u_1, \dots, u_i] = [[u_1, \dots, u_{i-1}], u_i]$,
where $u_1, \dots, u_i$ are arbitrary elements from $K\left\langle X \right\rangle$.
The commutator $[u_1, \dots, u_i]$ is called a commutator of length $i$ or simply an $i$-commutator.
Furthermore, we call a commutator \textit{pure} if it is of the form $[x_{i_1}, \dots, x_{i_j}]$ for some $x_{i_1}, \dots, x_{i_j} \in X$.

The lower central series of $K\left\langle X \right\rangle$ is the descending filtration
\[
L_1 \supseteq L_2 \supseteq \dots \supseteq L_p \supseteq \cdots,
\]
defined inductively by $L_1 = K\left\langle X \right\rangle$ and $L_p = [L_{p-1}, K\left\langle X \right\rangle]$ for $p \geq 2$.
Then $L_p$ is the Lie ideal in $K\left\langle X \right\rangle$ generated by all commutators of length $p$.
For $p \geq 2$, let $I_p = K\left\langle X \right\rangle \cdot L_p$ denote the two-sided associative ideal in $K\left\langle X \right\rangle$ generated by $L_p$.
In other words, $I_p$ is the T-ideal generated by the commutator $[x_1, \dots, x_p]$.
%The ideal $I_p$ is a \textit{T-ideal}, i.e., $I_p$ is closed under all $K$-algebra endomorphisms of $K\left\langle X \right\rangle$.
Following the traditions in the theory of groups and Lie algebras we denote by $\mathfrak{N}_p$ the variety of Lie nilpotent of class $\leq p$ associative algebras
(i.e., $\mathfrak{N}_p$ is the variety of associative algebras satisfying the polynomial identity $[x_1, \dots, x_{p+1}]= 0$).
%In other words, $\mathfrak{N}_p$ is the variety of Lie nilpotent algebras of Lie nilpotency index less or equal to $p$.
The algebra $K\left\langle X \right\rangle / I_{p+1}$ is the relatively free algebra of rank $n$ in the variety $\mathfrak{N}_{p}$.
We denote it in the standard way, i.e., $F_n(\mathfrak{N}_p)= K\left\langle X \right\rangle / I_{p+1}$.
When $p=1$, $F_n(\mathfrak{N}_1)$ is equal to the symmetric algebra $S(V)$ of $V$.
When $p=2$, $F_n(\mathfrak{N}_2)$ is the relatively free algebra of rank $n$ in the variety generated by the Grassmann algebra.

The group $\GL(n) := \GL(V, K)$ acts on the vector space $V$ with basis $X$ and this action is extended in a natural way
to the diagonal action of $\GL(n)$ on $K\left\langle X \right\rangle$.
Since for any $p$, $I_p$ is a $\GL(n)$-submodule of $K\left\langle X \right\rangle$, the algebra $F_n(\mathfrak{N}_p)$ is also a $\GL(n)$-module.
The $\GL(n)$-module structure of $F_n(\mathfrak{N}_p)$ is known for $p = 1, 2, 3, 4$ (see \cite{V} for $p=3$ and \cite{SV} for $p=4$).
Our first goal is to give some results on the $\GL(n)$-module structure of $F_n(\mathfrak{N}_p)$ for any $p$.
More precisely, let $\lambda = (\lambda_1, \lambda_2, \dots, \lambda_n) \in \ZZ^n_{\geq 0}$ with $\lambda_1 \geq \lambda_2 \geq \dots \geq \lambda_n \geq 0$
be a non-negative integer partition and let $V_{\lambda}$ denote the irreducible $\GL(n)$-module corresponding to the partition $\lambda$.
In Section \ref{sec_ModuleStructure}, we give a bound on the values of $\lambda$ such that $V_{\lambda}$ appears in the decomposition of $F_n(\mathfrak{N}_p)$ as a $\GL(n)$-module.
The main results in this direction are Theorem \ref{thm_bound} and Corollary \ref{coro_bound}.
The proofs are based on several known results on inclusions of products of commutators, which are recalled in Section \ref{sec_prelim},
and one new result on inclusions of products of commutators, which is stated in Section \ref{sec_ModuleStructure}.
In the end of Section \ref{sec_ModuleStructure}, we consider separately the case $n=3$.
As an application of our results on the $\GL(n)$-module structure of $F_n(\mathfrak{N}_p)$,
in Section \ref{sec_Invariants} we set $K = \CC$ and consider the algebra of invariants $F_n(\mathfrak{N}_p)^G$
where $G$ is one of the classical complex groups $\SL(n):=\SL(V, \CC)$, $\OO(n):=\OO(V,\CC)$, $\SO(n):=\SO(V, \CC)$, and $\Sp(2s):=\Sp(V,\CC)$ (the last in the case $n = 2s$).
By a theorem of Domokos and Drensky (\cite{DD}, \cite{DD2}), $F_n(\mathfrak{N}_p)^G$ is finitely generated. In Section \ref{sec_Invariants}, we give an upper bound on the degree of the generators of $F_n(\mathfrak{N}_p)^G$ in a minimal generating set (Corollary \ref{coro_beta}).
The advantage of our upper bound is that it is very explicit. We also give a general form for the Hilbert series $H(F_n(\mathfrak{N}_p)^G, t)$,
where $G$ is again one of $\SL(n)$, $\OO(n)$, $\SO(n)$, and $\Sp(2s)$, and show that $F_n(\mathfrak{N}_p)^{\SL(n)}$ and $F_n(\mathfrak{N}_p)^{\Sp(2s)}$ are finite-dimensional algebras.
In a similar way we consider also the algebra of invariants $F_{n}(\mathfrak{N}_p)^{G}$, where $G = \mathrm{UT}(n):=\mathrm{UT}(n, \CC)$ is the unitriangular group,
i.e., the subgroup of $\GL(n)$ consisting of upper triangular matrices with $1$'s on the diagonal.
It is known (see \cite{DD3}) that $F_{n}(\mathfrak{N}_p)^{\mathrm{UT}(n)}$ is finitely generated for any $p \geq 1$. In Section \ref{sec_Invariants} again, we give an upper bound on the degree of the generators in a minimal generating set (Corollary \ref{coro_beta_UT}).
The results from Section \ref{sec_Invariants} are then translated in the language of Classical Invariant Theory in Section \ref{sec_classInvTheory}.
In particular, for each $G$ as above, we give a criterion when a $G$-invariant of $T(V)$ belongs to $I_p$.
%Finally, in Section \ref{sec_rank3} we consider separately the case of relatively free algebras of rank $3$.
%In this particular case, since there are stronger results on inclusions of products of commutators the proof of Theorem \ref{thm_bound} can be made much shorter.
%Nevertheless, the bound on the values of $\lambda$ remains the same.

\section{Preliminaries} \label{sec_prelim}

In this section we summarize several results on inclusions of products of commutators, which will be extensively used in the next sections.
We state all results for a field of characteristic $0$, even though most of the results hold for more general fields.

The oldest result in this direction is the following lemma due to Jennings from 1947. %the following lemma due to Jennings (\cite{J}).

\begin{lemma}[\cite{J}]\label{lemma_commutators} Let $p \geq 3$. Then the product of any two $p$-commutators is an element of $I_{p+1}$.
\end{lemma}

\begin{proof}

Consider the product $[c_1, x][c_2, y]$ where $c_1$ and $c_2$ are arbitrary $p-1$-commutators (i.e., $c_1$ and $c_2$ belong to $L_{p-1}$)
and $x,y \in K\left\langle X \right\rangle$. We will show that $[c_1, x][c_2, y]$ can be expressed using elements of $I_{p+1}$.
We consider the expression

\begin{align*}
&[c_2, c_1y, x] = -[[c_1y, c_2], x] = -[c_1[y,c_2],x] -[[c_1, c_2]y, x] \\
&=c_1[c_2, y, x] + [c_1, x][c_2,y] -[c_1, c_2][y,x] - [c_1, c_2, x]y.
\end{align*}

Hence,
\begin{align}\label{eq_Jennings}
[c_1, x][c_2,y] = [c_2, c_1y, x] + [c_1, c_2][y,x] + [c_1, c_2, x]y - c_1[c_2, y, x].
\end{align}

Since $[c_1, c_2] \in L_{2p-2} \subset I_{2p-2}$ and for $p \geq 3$, $I_{2p-2}\subseteq I_{p+1}$ the lemma follows.
\end{proof}

Variations of the above proof lead to the following variations of the lemma of Jennings:

\begin{lemma} \label{lemma_diff_commutators1}
 Let $p \geq 3$ and let $ 3 \leq r \leq p$. The product of an $r$-commutator and a $p$-commutator is an element of $I_{p+1}$.
\end{lemma}

\begin{proof}
Let $c_1$ be an $r-1$-commutator and let $c_2$ be a $p-1$-commutator. Let $x$,$y$ be arbitrary elements of $K\left\langle X \right\rangle$. From the proof of Lemma \ref{lemma_commutators} we know that
\begin{align*}
[c_1, x][c_2,y] = [c_2, c_1y, x] + [c_1, c_2][y,x] + [c_1, c_2, x]y - c_1[c_2, y, x].
\end{align*}

Furthermore, $[c_1, c_2] \in L_{p+r-2} \subset I_{p+r-2}$. For $p \geq 3$ and $ 3\leq r \leq p$, we have $I_{p+r-2}\subseteq I_{p+1}$, hence the lemma follows.
\end{proof}

\begin{lemma} \label{lemma_diff_commutators2}
 %Let $p \geq 2$. Let $c_1$ be an $r-1$-commutator for $ 2\leq r \leq p-2$ and let $c_2$ be a $p-1$-commutator. Let $y$ be an arbitrary element of $K\left\langle X \right\rangle$. Then
Let $p \geq 2$ and let $2\leq r \leq p$. Let $y$ be an arbitrary element of $K\left\langle X \right\rangle$. Let $[c_1, y] \in L_{r}$ and $[c_2, y] \in L_p$. Then,
\[
[c_1, y][c_2,y] \in I_{p+1}.
\]

In particular, for $x, y, z_1, \dots, z_{p-1} \in K\left\langle X \right\rangle$ we have
\[
[x,y][z_1, \dots, z_{p-1}, y] \in I_{p+1}.
\]
\end{lemma}

\begin{proof}
From Equation(\ref{eq_Jennings}) we obtain that
\begin{align*}
[c_1, y][c_2,y] = [c_2, c_1y, y] + [c_1, c_2][y,y] + [c_1, c_2, y]y - c_1[c_2, y, y],
\end{align*}
which proves the statement.
\end{proof}

The next result improves considerably Lemma \ref{lemma_commutators} and Lemma \ref{lemma_diff_commutators1} (but not Lemma \ref{lemma_diff_commutators2}). It was first proved by Latyshev in 1965 and then was independently rediscovered by Gupta and Levin in 1983.

\begin{theorem}[\cite{L}, \cite{GL}] \label{thm_product1}
For all $m_1$ and $m_2$
\[
I_{m_1}I_{m_2} \subseteq I_{m_1 + m_2 - 2}.
\]
\end{theorem}

When one of the integers $m_1$ and $m_2$ is odd, an even stronger inclusion holds. The following theorem was proved several times with different restrictions on the characteristic of $K$. Here we cite only two of the sources. The most general result (i.e., for most general fields $K$) was proved in \cite{DeK2}.

\begin{theorem}[\cite{BJ}, \cite{DeK2}] \label{thm_product2}
If $m_1$ or $m_2$ is odd then
\[
I_{m_1}I_{m_2} \subseteq I_{m_1 + m_2 - 1}.
\]
\end{theorem}

%It is worth pointing out that the condition in Theorem \ref{thm_product2} is not only sufficient but also necessary. Namely, the following theorem was proven in \cite{DeK1}:

%\begin{theorem}[\cite{DeK1}] \label{thm_product4}
%Let $m_1$ and $m_2$ be arbitrary positive even integers. Then
%\[
%I_{m_1}I_{m_2} \nsubseteq I_{m_1 + m_2 - 1}.
%\]
%\end{theorem}

The proof of Theorem \ref{thm_product2} given in \cite{BJ} is based on the following statement, which is of interest of its own. %which we will also need in what follows.

\begin{theorem}[\cite{BJ}] \label{thm_Lie_ideals}
\[
[L_{m_1}, I_{m_2}] \subseteq L_{m_1 + m_2},
\]
whenever $m_2$ is odd.
\end{theorem}

The next statement is derived from Theorems \ref{thm_product1} and \ref{thm_product2} and can be found in \cite{Da} and \cite{DeK2}.

\begin{theorem}  \label{thm_product3}
For every $l \geq 2$ and for every choice of positive integers $m_1, \dots, m_l$ one of the following holds:
\begin{itemize}
\item [(i)] If $k$ indices among $m_1, \dots, m_l$ are odd with $0 \leq k < l$ then
\[
I_{m_1}\cdots I_{m_l} \subseteq I_{m_1 + \dots + m_l -2l+k+2}.
\]
The number $m_1 + \dots + m_l -2l+k+2$ is always even.
\item[(ii)] If all indices are odd then
\[
I_{m_1}\cdots I_{m_l} \subseteq I_{m_1 + \dots + m_l -l + 1}.
\]
The number $m_1 + \dots + m_l -l + 1$ is always odd.
\end{itemize}
\end{theorem}

In the end of this section, let us fix $n=3$, i.e., let us consider the free associative algebra $K\left \langle x_1, x_2, x_3\right\rangle$. Then the following result holds:

\begin{theorem}[\cite{P}]\label{thm_P}
Consider the algebra $K\left \langle x_1, x_2, x_3\right\rangle$. Then
\[
I_{m_1}I_{m_2} \subseteq I_{m_1+m_2-1},
\]
for all integers $m_1, m_2 \geq 2$.
\end{theorem}

%The result of Pchelintsev for $n=3$ is independent of the parity of $m_1$ and $m_2$ and thus is stronger than the general Theorems \ref{thm_product1} and \ref{thm_product2}.
It is worth pointing out, that the restriction $n=3$ in Theorem \ref{thm_P} is essential. It was proved in \cite{DeK1} that in general for even integers $m_1$ and $m_2$ we have
\[
I_{m_1}I_{m_2} \nsubseteq I_{m_1+m_2-1}.
\]

\section{The $\GL(n)$-module structure of $F_n(\mathfrak{N}_p)$} \label{sec_ModuleStructure}

In this section we discuss the $\GL(n)$-module structure of $F_n(\mathfrak{N}_p)$ for any $p \geq 1$. First we introduce some notations. Let $B_n$ denote the subalgebra of $K\left\langle X \right\rangle$ generated by all pure commutators $[x_{i_1},\dots, x_{i_r}]$ with $r \geq 2$. The elements of $B_n$ are called proper polynomials. Clearly, $B_n$ is a $\GL(n)$-submodule of $K\left\langle X \right\rangle$. Let $\varphi_p:K\left\langle X \right\rangle \rightarrow F_n(\mathfrak{N}_p)$ denote the natural surjective $K$-algebra homomorphism. Then $\varphi_p(B_n)$ is a $\GL(n)$-submodule of $F_n(\mathfrak{N}_p)$. By a theorem of Drensky (\cite{D}), for any $p \geq 1$ we have the following decomposition as $\GL(n)$-modules:
\begin{align} \label{eq_decomp}
 F_n(\mathfrak{N}_p) \cong_{\GL(n)} S(V) \otimes \varphi_p(B_n),
\end{align}
where again $S(V)$ denotes the symmetric algebra of $V$.

We start by proving the following lemma, which improves considerably Lemma \ref{lemma_diff_commutators2} from Section \ref{sec_prelim}.

\begin{lemma} \label{lemma_even_product}
Let $m_1$ and $m_2$ be even integers. Let $[c_1, y] \in L_{m_1}$ and $[c_2, y] \in L_{m_2}$, where $y$ is an arbitrary element from $K \left\langle X \right\rangle$. Then,
\[
[c_1, y][c_2, y] \in I_{m_1 + m_2 -1}.
\]
\end{lemma}

\begin{proof}
From Equation (\ref{eq_Jennings}) we have that
\begin{align*}
[c_1, y][c_2,y] = [c_2, c_1y, y] + [c_1, c_2][y,y] + [c_1, c_2, y]y - c_1[c_2, y, y].
\end{align*}
We need to show that all non-zero terms on the right side of the above equation belong to $I_{m_1 + m_2 -1}$. This clearly holds for $[c_1, c_2, y]y$. Consider the term $c_1[c_2, y, y]$. $c_1$ is a commutator of odd length $m_1 -1$, whereas $[c_2, y, y]$ is a commutator of odd length $m_2 +1$. By Theorem \ref{thm_product2}, $c_1[c_2, y, y] \in I_{m_1 + m_2 -1}$.
It remains to consider the term $[c_2, c_1y, y]$. We have that $c_2 \in L_{m_2 -1}$ and $c_1y \in I_{m_1 -1}$. By Theorem \ref{thm_Lie_ideals},
\[
[c_2, c_1y] \in L_{m_1 + m_2 -2}.
\]
This completes the proof.

\end{proof}

\begin{corollary} \label{coro_even_products}
Let $m_1, \dots, m_l$ be even integers, let $y \in K\left\langle X \right\rangle$, and let $[c_i, y] \in L_{m_i}$ for $i = 1, \dots, l$. Then
\[
[c_1, y] \cdots [c_l, y] \in I_{m_1 + \dots + m_l -l +1}.
\]
\end{corollary}

\begin{proof}
If $l$ is even, for $k=1, \dots, \frac{l}{2}$, Lemma \ref{lemma_even_product} implies that $[c_{2k-1}, y][c_{2k}, y] \in I_{m_{2k-1} + m_{2k} -1}$. Moreover, the number $m_{2k-1} + m_{2k} -1$ is odd. Let us denote it by $j_k$. Therefore, by Theorem \ref{thm_product3}(ii) we obtain that
\[
([c_1, y][c_2, y]) \cdots ([c_{l-1}, y][c_l, y]) \in I_{j_1}\cdots I_{j_{l/2}} \subseteq I_{j_1 + \dots + j_{l/2} - l/2 +1} = I_{m_1 + \dots + m_l -l +1}.
\]

Next, if $l$ is odd, we can apply the above considerations for $l-1$. We obtain that
\[
[c_1, y] \cdots [c_{l-1}, y] \in I_{m_1 + \dots + m_{l-1} -l +2}.
\]

Furthermore, $m_1 + \dots + m_{l-1} -l +2$ is an odd number. Hence, by Theorem \ref{thm_product2},
\[
[c_1, y] \cdots [c_{l-1}, y][c_l, y] \in I_{m_1 + \dots + m_{l-1} -l +2}I_{m_l} \subseteq I_{m_1 + \dots + m_l - l + 1}.
\]
\end{proof}

We now give the following estimate for the $\GL(n)$-module structure of $\varphi_p(B_n)$ for any $p \geq 1$.

\begin{theorem}\label{thm_bound}
 Let $p \geq 1$. Suppose that
\[
\varphi_p(B_n) \cong_{\GL(n)} \bigoplus_{\lambda} k_{\lambda} V_{\lambda},
\]
where again $V_{\lambda}$ denotes the irreducible $\GL(n)$-module corresponding to the non-negative integer partition $\lambda = (\lambda_1,\lambda_2, \dots, \lambda_n)$. %Then for even $p$ we have $\lambda_1 \leq \frac{3p-6}{2}$ and for odd $p$, $\lambda_1 \leq \frac{3p-7}{2}$.
Then $\lambda_1 \leq  p-1$.
\end{theorem}

\begin{proof}
For $p=1$ the statement is trivially true. Therefore, we give a proof for the case $p \geq 2$. The algebra $K\left\langle X \right\rangle$ has a $\ZZ_{\geq 0}^n$-grading given by the weight space decomposition as a $\GL(n)$-module. More precisely,
\[
K\left\langle X \right\rangle = \bigoplus_{\mu = (\mu_1, \dots, \mu_n) \in \ZZ_{\geq 0}^n} W(\mu),
\]
where $W(\mu)$ is the weight space corresponding to the weight $\mu$. This means that, if $g = \mathrm{diag} (\xi_1, \dots, \xi_n) \in \GL(n)$ and $w \in W(\mu)$, then
\[
g(w) = \xi_1^{\mu_1} \cdots \xi_n^{\mu_n}w.
\]
Therefore, a basis for $W(\mu)$ is given by the monomial $x_1^{\mu_1}\cdots x_n^{\mu_n}$ and all its distinct permutations.

Consequently, $B_n$ and $\varphi_p(B_n)$ also have $\ZZ_{\geq 0}^n$-gradings given respectively by
\[
B_n = \bigoplus_{\mu = (\mu_1, \dots, \mu_n) \in \ZZ_{\geq 0}^n} W(\mu) \cap B_n;
\]
\[
\varphi_p(B_n) = \bigoplus_{\mu = (\mu_1, \dots, \mu_n) \in \ZZ_{\geq 0}^n} \varphi_p(W(\mu) \cap B_n).
\]

Notice that $\varphi_p(W(\mu) \cap B_n)$ consists of sums of products of pure commutators such that the variable $x_i$ appears in each product exactly $\mu_i$ times for all $i = 1, \dots, n$.

Consider now a partition $\lambda = (\lambda_1 \geq \lambda_2 \geq \dots \geq \lambda_n)$ such that $V_{\lambda}$ appears in the decomposition of $\varphi_p(B_n)$. Therefore, $\varphi_p(W(\lambda) \cap B_n) \neq \emptyset$ and hence $\varphi_p(B_n)$ contains an element $w$ such that $w$ is a sum of products of pure commutators and the variable $x_1$ appears in each product exactly $\lambda_1$ times. We will estimate the maximal possible number of appearances of $x_1$ in an element of $\varphi_p(B_n)$. For brevity, let us denote this maximal number by $N$.

Let $u$ be a product of pure commutators, such that the number of appearances of $x_1$ in $u$ is equal to $N$. We notice the following property, which was also considered in \cite{Da}. If $u \in I_k$ for some $2 \leq k \leq p$ and if we permute the commutators which participate in $u$ then we obtain again a product of pure commutators which belongs to $I_k$. This follows from the fact that for all $i, j \geq 1$, $[L_i, L_j] \subseteq L_{i+j}$. Therefore, without loss of generality, we may consider that $u = u_1u_2u_3$, such that the following conditions are fulfilled: $u_1$ is a product of $l_1$ pure commutators of lengths respectively $m_1, \dots, m_{l_1}$, such that each $m_i$ is odd; $u_2$ is a product of $l_2$ pure commutators of lengths respectively $n_1, \dots, n_{l_2}$, such that all $n_i$ are even and the last position in each commutator is equal to $x_1$; $u_3$ is a product of $l_3$ pure commutators of lengths respectively $r_1, \dots, r_{l_3}$, such that all $r_i$ are even and the last position in each commutator is different from $x_1$. Let us denote the maximal number of appearances of $x_1$ in $u_i$ by $N_i$, for $i=1,2,3$. We need to estimate $N_1$, $N_2$, and $N_3$.

By Theorem \ref{thm_product3} (ii), $u_1 \in I_{m_1 + \dots + m_{l_1} -l_1 +1}$. We set $p_1 = m_1 + \dots + m_{l_1} -l_1+1$. The maximal number of appearances of $x_1$ in $u_1$ is $N_1 = m_1 + \dots +m_{l_1} -l_1$ (since in a pure $m_i$-commutator the variable $x_1$ can appear at most $m_i-1$ times). Therefore, $N_1 = p_1 -1$.

Next, we consider the element $u_2$. By Corollary \ref{coro_even_products}, $u_2 \in I_{n_1 + \dots + n_{l_2} -{l_2} +1}$. Similarly as above, we set $p_2 = n_1 + \dots + n_{l_2} -{l_2}+1$ and we obtain that $N_2 = p_2 -1$.

Finally, we take $u_3$. By Theorem \ref{thm_product3} (i), we have that $u_3 \in I_{r_1 + \dots + r_{l_3} - 2l_3 + 2}$. We set $p_3 = r_1 + \dots + r_{l_3} - 2l_3 + 2$. Since in each of the commutators in $u_3$ the last position is different from $x_1$ we have that the maximal number of appearances of $x_1$ in $u_3$ is $N_3 = r_1 + \dots +r_{l_3} -2l_3$. Hence, $N_3 = p_3 -2$.

In short, the above implies that
\[
u = u_1u_2u_3 \in I_{p_1}I_{p_2}I_{p_3} \subseteq I_p.
\]

Notice that $p_1$ is always an odd number, whereas $p_3$ is always an even number. Therefore, Theorem \ref{thm_product2} implies that
\[
I_{p_1}I_{p_2} \subseteq I_{p_1 + p_2 -1}.
\]

Hence, by Theorem \ref{thm_product1} we obtain that
\[
I_{p_1}I_{p_2}I_{p_3} \subseteq I_{p_1 + p_2 -1}I_{p_3} \subseteq I_{p_1 + p_2 + p_3 -3} \subseteq I_p.
\]
Therefore, $p_1 + p_2 + p_3 - 3 \leq p$. Thus, $N = N_1 + N_2 + N_3 = p_1 -1 + p_2 -1 + p_3 -2 \leq p-1$. This proves the statement.

\end{proof}

Theorem \ref{thm_bound} together with Equation (\ref{eq_decomp}) lead to the following corollary.

\begin{corollary}\label{coro_bound}
Let $p \geq 1$. Suppose that
\[
F_n(\mathfrak{N}_p) \cong_{\GL(n)} \bigoplus_{\lambda}m_{\lambda} V_{\lambda},
\]
where again $V_{\lambda}$ denotes the irreducible $\GL(n)$-module corresponding to the non-negative integer partition $\lambda = (\lambda_1, \lambda_2, \dots, \lambda_n)$. Then $\lambda_2 \leq p-1$.
\end{corollary}

\begin{proof}
For $p=1$ the statement is again trivially true. For $p \geq 2$, the proof is a standard task on branching rules. We apply Pieri's branching rule (also known as Young's rule) to Equation (\ref{eq_decomp}) and then we use Theorem \ref{thm_bound}. One may read the description of Pieri's rule in e.g. \cite{FH}.
% We apply Pieri's branching rule (also known as Young's rule) to the decomposition of $F_n(\mathfrak{N}_p)$ as a $\GL(n)$-module given in Equation (\ref{eq_decomp}).
\end{proof}

We notice that the upper bound on $\lambda_1$ which is obtained in Theorem \ref{thm_bound} is exact for all known cases. More precisely, using the decomposition of $\varphi_p(B_n)$ over $\GL(n)$, which is known for $p = 2,3,4$, we see that for $p=2$, $\lambda_1 \leq 1 = p-1$ and for $p=3$, $\lambda_1 \leq 2 = p-1$. For $p=4$, we have that $\lambda_1 \leq 3 = p-1$.

For $p=5$, Theorem \ref{thm_bound} gives that $\lambda_1 \leq 4$. \\

In the end of the section, we would like to consider separately the case $n=3$, i.e., we take the free associative algebra $K\left \langle x_1, x_2, x_3\right\rangle$ and the relatively free algebra $F_3(\mathfrak{N}_p)$ of rank $3$ in the variety $\mathfrak{N}_p$. Then, we have the stronger theorem on inclusions of products of commutators (Theorem \ref{thm_P}) which leads to a much shorter proof of Theorem \ref{thm_bound}. For completeness of the exposition, we present this proof below.

{\it Shorter proof of Theorem \ref{thm_bound} for $n=3$.}

Consider a partition $\lambda$ such that $V_{\lambda}$ appears in the decomposition of $\varphi_p(B_3)$ as a $\GL(3)$-module. Therefore, as in the general proof of Theorem \ref{thm_bound}, $\varphi_p(B_3)$ contains an element $u$ such that $u$ is a product of pure commutators and the variable $x_1$ appears in $u$ exactly $\lambda_1$ times. We will again estimate the maximal number of appearances of the variable $x_1$ in $u$. By Theorem \ref{thm_P}, for all integers $m_1, \dots, m_l \geq 2$
\[
 I_{m_1}I_{m_2} \cdots I_{m_l} \subseteq I_{m_1 + m_2 + \dots + m_l - l + 1}.
\]
Hence, $u$ is at most a product of an $m_1$-commutator, $m_2$-commutator, $\dots$, and $m_l$-commutator such that $m_1 + \dots + m_l - l + 1 = p$. Hence $m_1 + \dots + m_l = p+l-1$. The number of appearances of $x_1$ in such a product of pure commutators $u$ is at most $p-1$. %(since in a pure $m_i$-commutator the variable $x_1$ can appear at most $m_i - 1$ times).
This proves the statement.
%\end{proof}

%It is worth mentioning, that even though for $n=3$ there are stronger results on inclusions of products of commutators, the bound on $\lambda_1$ remains the same. This might indicate that the bound obtained in Theorem \ref{thm_bound} is optimal and cannot be improved further.

\section{The algebra of invariants $F_n(\mathfrak{N}_p)^G$ for $G = \SL(n)$, $\OO(n)$, $\SO(n)$, $\Sp(2s)$, and $\mathrm{UT}(n)$} \label{sec_Invariants}

In this section we set $K = \CC$. Using results from \cite{BBDGK} and \cite{DH}, we will obtain some bounds on the generators of the algebra of invariants $F_n(\mathfrak{N}_p)^G$ for $G = \SL(n)$, $\OO(n)$, $\SO(n)$, $\Sp(2s)$, and $\mathrm{UT}(n)$.

\begin{proposition} \label{prop_HilbertSeries_SLSP}
Let $p \geq 1$ and let $G = \SL(n)$ or $G = \Sp(2s)$ (for $n=2s$). Then, the algebra of invariants $F_n(\mathfrak{N}_p)^G$ is finite-dimensional. More precisely, the Hilbert series $H(F_n(\mathfrak{N}_p)^G, t)$ is a polynomial and it holds that
\begin{align*}
\deg H(F_n(\mathfrak{N}_p)^G, t) \leq n(p-1),
\end{align*}
where in the case $G = \Sp(2s)$ we have that $n = 2s$.

\end{proposition}

\begin{proof}

We denote by $F_n^{(i)}(\mathfrak{N}_p)$ the homogeneous component of total degree $i$ in $F_n(\mathfrak{N}_p)$. Then we can write the decomposition of $F_n(\mathfrak{N}_p)$ as a $\GL(n)$-module in the following way:
\[
F_n(\mathfrak{N}_p) = \bigoplus_{i \geq 0} F_n^{(i)}(\mathfrak{N}_p) \cong_{\GL(n)} \bigoplus_{i \geq 0} \bigoplus_{\lambda} m_{i, \lambda} V_{\lambda},
\]
where $|\lambda| = i$ and, by Corollary \ref{coro_bound}, $\lambda_2 \leq p-1$.

Consider first the case when $G = \Sp(2s)$. Theorem 2.2 from \cite{DH} implies that
\[
H(F_{2s}(\mathfrak{N}_p)^{\Sp(2s)}, t) = \sum_{i \geq 0}\left(\sum_{\lambda} m_{i, \lambda}\right)t^i,
\]
where the second sum runs over partitions $\lambda$ with $\lambda_1 = \lambda_2$, $\dots$, $\lambda_{2s-1} = \lambda_{2s}$ such that $|\lambda| = i$ and $\lambda_2 \leq p-1$. The largest partition $\lambda$ which satisfies these conditions is $\lambda_1 = \lambda_2 = \dots = \lambda_{2s-1} = \lambda_{2s} = p-1$. Therefore, the highest power of $t$ which can occur in the Hilbert series $H(F_{2s}(\mathfrak{N}_p)^{\Sp(2s)}, t)$ is $i = \lambda_1 + \dots + \lambda_{2s} = 2s(p-1)$. This proves the statement for $G = \Sp(2s)$.

Next, let $G = \SL(n)$. The results from \cite{BBDGK} together with Corollary \ref{coro_bound} imply that
\[
H(F_{n}(\mathfrak{N}_p)^{\SL(n)}, t) = \sum_{i \geq 0}\left(\sum_{\lambda} m_{i, \lambda}\right)t^i,
\]
where the second sum runs over partitions $\lambda$ with $\lambda_1 = \lambda_2 = \dots = \lambda_{n}$ such that $|\lambda| = i$ and $\lambda_2 \leq p-1$. Hence, the largest partition $\lambda$ which satisfies these conditions is again $\lambda_1 = \lambda_2 = \dots = \lambda_{n} = p-1$. Therefore, the highest power of $t$ which can occur in the Hilbert series $H(F_{n}(\mathfrak{N}_p)^{\SL(n)}, t)$ is again $i = n(p-1)$. This completes the proof.
\end{proof}

\begin{proposition} \label{prop_HilbertSeries_O} Let $p \geq 1$ and let $G = \OO(n)$. Then,
\[
H(F_{n}(\mathfrak{N}_p)^{\OO(n)}, t) = \frac{p(t)}{1-t^2},
\]
where $p(t)$ is a polynomial such that
\begin{align*}
\deg p(t) \leq
\left \{ \begin{array} {ll}
n(p-1) & \text{if }  p-1 \text{ is even}\\
n(p - 2) & \text{if } p-1 \text{ is odd}.
\end{array} \right.
\end{align*}
\end{proposition}

\begin{proof}
We proceed as in the proof of the previous proposition. We write the decomposition of $F_n(\mathfrak{N}_p)$ as a $\GL(n)$-module in the following way:
\[
F_n(\mathfrak{N}_p) = \bigoplus_{i \geq 0} F_n^{(i)}(\mathfrak{N}_p) \cong_{\GL(n)} \bigoplus_{i \geq 0} \bigoplus_{\lambda} m_{i, \lambda} V_{\lambda},
\]
where again $|\lambda| = i$ and $\lambda_2 \leq p-1$.
Then Theorem 2.2 from \cite{DH} implies that
\[
H(F_{n}(\mathfrak{N}_p)^{\OO(n)}, t) = \sum_{i \geq 0}\left(\sum_{\lambda} m_{i, \lambda}\right)t^i,
\]
where the second sum runs over even partitions $\lambda$ such that $|\lambda| = i$ and $\lambda_2 \leq p-1$.

If a partition $\lambda' = (\lambda_1', \lambda_2', \dots, \lambda_n')$ satisfies the above conditions and appears in the decomposition of $F_{n}(\mathfrak{N}_p)^{\OO(n)}$, then due to Equation (\ref{eq_decomp}) all partitions of the form $(\lambda_1' + 2k, \lambda_2', \dots, \lambda_n')$ for $k = 1, 2, \dots$ also appear in the decomposition of $F_{n}(\mathfrak{N}_p)^{\OO(n)}$. Thus, the following term appears in $H(F_{n}(\mathfrak{N}_p)^{\OO(n)}, t)$:
\[
t^{|\lambda'|}(1 + t^2 + t^4 + \dots) = \frac{t^{|\lambda'|}}{1-t^2}.
\]

We can write all terms in $H(F_{n}(\mathfrak{N}_p)^{\OO(n)}, t)$ in this way for a suitable $\lambda'$. Theorem \ref{thm_bound} implies that when $p-1$ is even, the largest suitable $\lambda'$ that satisfies the above conditions is $\lambda' = (p-1, p-1, \dots, p-1)$.
Similarly, when $p-1$ is odd, the largest suitable $\lambda'$ that satisfies the above conditions is $\lambda' = (p-2, p-2, \dots, p-2)$. This proves the statement.
\end{proof}

In the same way we can prove the following proposition.

\begin{proposition} \label{prop_HilbertSeries_SO} Let $p \geq 1$ and let $G = \SO(n)$. Then,
\[
H(F_{n}(\mathfrak{N}_p)^{\SO(n)}, t) = \frac{p(t)}{1-t^2},
\]
where $p(t)$ is a polynomial such that
\begin{align*}
\deg p(t) \leq n(p-1).
\end{align*}
\end{proposition}

By a theorem of Domokos and Drensky (\cite{DD}), the algebra of invariants $F_{n}(\mathfrak{N}_p)^{G}$ is finitely generated for any reductive subgroup $G$ of $\GL(n)$. Using the above results we can give an explicit upper bound for the degree of the generators in a minimal generating set of $F_{n}(\mathfrak{N}_p)^{G}$ for $G = \SL(n)$, $\OO(n)$, $\SO(n)$, or $\Sp(2s)$. First, we introduce the following notation. Given a graded algebra
$A = \bigoplus_{i \geq 0} A^i,
$
we denote by $\beta(A)$ the minimal non-negative integer $i$ such that $A$ is generated by homogeneous elements of degree at most $i$. We write $\beta(A) = \infty$ if there is no such $i$.

\begin{corollary} \label{coro_beta}
Let $G$ be one of $\SL(n)$ or $\Sp(2s)$. Then
\begin{align*}
\beta(F_{n}(\mathfrak{N}_p)^{G}) \leq n(p-1), \text{ for all } p \geq 1.
\end{align*}

For $G = \OO(n)$ we have that
\begin{align*}
\beta(F_{n}(\mathfrak{N}_p)^{\OO(n)}) \leq
\left \{ \begin{array} {ll}
n(p-1) & \text{if } p-1 \text{ is even and }  p > 1\\
n(p - 2) & \text{if } p-1 \text{ is odd and } p > 2\\
2 & \text{if } p=1,2.
\end{array} \right.
\end{align*}

For $G = \SO(n)$ we have that
\begin{align*}
\beta(F_{n}(\mathfrak{N}_p)^{\SO(n)}) \leq
\left \{ \begin{array} {ll}
n(p-1) & \text{if } p > 1\\
2 & \text{if } p=1.
\end{array} \right.
\end{align*}
\end{corollary}

\begin{proof}
By Proposition \ref{prop_HilbertSeries_SLSP}, the statement is clear for $G = \SL(n)$ and $G = \Sp(2s)$. Hence,  we consider the case of $G$ being one of $\OO(n)$ and $\SO(n)$. By Propositions \ref{prop_HilbertSeries_O} and \ref{prop_HilbertSeries_SO}
\[
H(F_{n}(\mathfrak{N}_p)^{G}, t) = \frac{p(t)}{1-t^2} = p(t)(1 + t^2 + t^4 + \dots + t^{2k} + \dots).
\]

$F_{n}(\mathfrak{N}_p)^{G}$ has one free generator of degree $2$, namely $u = x_1^2 + \dots + x_n^2$. Therefore, if $\{f_1, \dots, f_l \}$ is a basis for the homogeneous part $\bigoplus_{i = 0}^{\deg p(t)} F_{n}^{(i)}(\mathfrak{N}_p)^{G}$ then $\{f_1, \dots, f_l, u \}$ is a generating set (not necessarily minimal) for the whole $F_{n}(\mathfrak{N}_p)^{G}$. This proves the statement.
\end{proof}

\begin{remark} In \cite{DD2}, Domokos and Drensky give a general bound for $\beta(F_{n}(\mathfrak{N}_p)^{G})$ for any reductive group $G$ as a special case of a more general statement. The advantage of the bound that we find in Corollary \ref{coro_beta} is that it is very explicit and can be used in explicit computations.
\end{remark}

In a similar way we can consider the algebra of invariants $F_{n}(\mathfrak{N}_p)^{\mathrm{UT}(n)}$, where $\mathrm{UT}(n):=\mathrm{UT}(n, \CC)$ denotes the unitriangular group, i.e., the subgroup of $\GL(n)$ consisting of upper triangular matrices with $1$'s on the diagonal. $\mathrm{UT}(n)$ is a maximal unipotent subgroup of $\GL(n)$ and of $\SL(n)$. Using results from \cite{BBDGK} we prove the following proposition.

\begin{proposition} \label{prop_HilbertSeries_UT} Let $p \geq 1$ and let $G = \mathrm{UT}(n)$. Then,
\[
H(F_{n}(\mathfrak{N}_p)^{\mathrm{UT}(n)}, t) = \frac{p(t)}{1-t},
\]
where $p(t)$ is a polynomial such that
\begin{align*}
\deg p(t) \leq n(p-1).
\end{align*}
\end{proposition}

\begin{proof}
We proceed as before. We write the decomposition of $F_n(\mathfrak{N}_p)$ as a $\GL(n)$-module in the following way:
\[
F_n(\mathfrak{N}_p) = \bigoplus_{i \geq 0} F_n^{(i)}(\mathfrak{N}_p) \cong_{\GL(n)} \bigoplus_{i \geq 0} \bigoplus_{\lambda} m_{i, \lambda} V_{\lambda},
\]
where again $|\lambda| = i$ and $\lambda_2 \leq p-1$. Each module $V_{\lambda}$ has a one-dimensional $\mathrm{UT}(n)$-invariant subspace generated by any highest weight vector. Therefore (see also \cite{BBDGK}),
\[
H(F_{n}(\mathfrak{N}_p)^{\mathrm{UT}(n)}, t) = \sum_{i \geq 0}\left(\sum_{\lambda} m_{i, \lambda}\right)t^i,
\]
where the second sum runs over partitions $\lambda$ such that $|\lambda| = i$ and $\lambda_2 \leq p-1$.

Then, as in the proof of Proposition \ref{prop_HilbertSeries_O}, we notice the following. If a partition $\lambda' = (\lambda_1', \lambda_2', \dots, \lambda_n')$ satisfies the above conditions and appears in the decomposition of $F_{n}(\mathfrak{N}_p)^{\mathrm{UT}(n)}$, then due to Equation (\ref{eq_decomp}) all partitions of the form $(\lambda_1' + k, \lambda_2', \dots, \lambda_n')$ for $k = 1, 2, \dots$ also appear in the decomposition of $F_{n}(\mathfrak{N}_p)^{\mathrm{UT}(n)}$. Thus, the following term appears in $H(F_{n}(\mathfrak{N}_p)^{\mathrm{UT}(n)}, t)$:
\[
t^{|\lambda'|}(1 + t + t^2 + t^3 + \dots) = \frac{t^{|\lambda'|}}{1-t}.
\]

We can write all terms in $H(F_{n}(\mathfrak{N}_p)^{\mathrm{UT}(n)}, t)$ in this way for a suitable $\lambda'$. Theorem \ref{thm_bound} implies that the largest suitable $\lambda'$ that satisfies the above conditions is $\lambda' = (p-1, p-1, \dots, p-1)$. This proves the statement.
\end{proof}

The algebra $F_{n}(\mathfrak{N}_p)$ is left Noetherian, hence Theorem 7.2 from \cite{DD3} implies that the algebra of invariants $F_{n}(\mathfrak{N}_p)^{\mathrm{UT}(n)}$ is finitely generated. Using Proposition \ref{prop_HilbertSeries_UT} we can give an upper bound on the degree of the generators of $F_{n}(\mathfrak{N}_p)^{\mathrm{UT}(n)}$ in a minimal generating set.

\begin{corollary} \label{coro_beta_UT}
For the algebra of invariants $F_{n}(\mathfrak{N}_p)^{\mathrm{UT}(n)}$ it holds that
\begin{align*}
\beta(F_{n}(\mathfrak{N}_p)^{\mathrm{UT}(n)}) \leq
\left \{ \begin{array} {ll}
n(p-1) & \text{if } p > 1\\
1 & \text{if } p=1.
\end{array} \right.
\end{align*}
\end{corollary}

\begin{proof}
We notice that $x_1$ is a free generator of $F_{n}(\mathfrak{N}_p)^{\mathrm{UT}(n)}$ of degree $1$. Furthermore, by Proposition \ref{prop_HilbertSeries_UT}
\[
H(F_{n}(\mathfrak{N}_p)^{\mathrm{UT}(n)}, t) = \frac{p(t)}{1-t} = p(t)(1 + t + t^2 + t^3 + \dots).
\]
Therefore, as in the proof of Corollary \ref{coro_beta}, if $\{f_1, \dots, f_l \}$ is a basis for the homogeneous part $\bigoplus_{i = 0}^{\deg p(t)} F_{n}^{(i)}(\mathfrak{N}_p)^{\mathrm{UT}(n)}$ then $\{f_1, \dots, f_l, x_1 \}$ is a generating set (not necessarily minimal) for the whole $F_{n}(\mathfrak{N}_p)^{\mathrm{UT}(n)}$. This proves the statement.
\end{proof}

\begin{remark} Notice that for $p=2$, using results from \cite{DH}, we can compute explicitly $\beta(F_{n}(\mathfrak{N}_p)^{G})$ for $G = \SL(n)$, $\OO(n)$, $\SO(n)$, $\Sp(2s)$, and $\mathrm{UT}(n)$. We see that the bounds obtained in Corollary \ref{coro_beta} and \ref{coro_beta_UT} for $p=2$ are exact when $G = \SL(n)$, $\OO(n)$, $\SO(n)$, and $\mathrm{UT}(n)$. Namely, $\beta(F_{n}(\mathfrak{N}_2)^{\SL(n)}) = \beta(F_{n}(\mathfrak{N}_2)^{\SO(n)}) = \beta(F_{n}(\mathfrak{N}_2)^{\mathrm{UT}(n)}) = n = n(p-1)$ and $\beta(F_{n}(\mathfrak{N}_2)^{\OO(n)}) = 2$. For $G = \Sp(2s)$, we obtain $\beta(F_{2s}(\mathfrak{N}_2)^{\Sp(2s)}) = 2 < 2s(p-1)$ for $s >1$, i.e., when $G = \Sp(2s)$ the bound from Corollary \ref{coro_beta} is not exact for $p=2$.
\end{remark}

\section{Applications to Classical Invariant Theory} \label{sec_classInvTheory}

In this section we will reformulate Corollaries \ref{coro_beta} and \ref{coro_beta_UT} in the language of Classical Invariant Theory. We set again $K = \CC$. First we recall a classical result due to Weyl involving the algebra of invariants $\CC \left \langle X \right\rangle^G = T(V)^G$, where $G$ is one of $\OO(n)$ or $\Sp(2s)$. In order to state this result, we follow the notations and formulations in \cite{H}. Let $\left \langle \cdot, \cdot \right\rangle$ denote the non-degenerate symmetric (in the case of $G = \OO(n)$) or skew symmetric  (in the case of $G = \Sp(2s))$ form on $V$ preserved by the action of $G$. Using the natural identification of $V^*$ with $V$ we obtain that the form $\left \langle \cdot, \cdot \right\rangle$ is an invariant for $G$ in $V \otimes V = V^{\otimes 2}$. Following \cite{H}, we denote this invariant by $\theta_2$. For $k = 2j$ we may choose an isomorphism
\[
i_k : V^{\otimes k} \cong \otimes^j V^{\otimes 2}.
\]
Specifying $i_k$ amounts to pairing off the factors of $V^{\otimes k}$. Then, in $\otimes^j V^{\otimes 2}$ we have the invariant $\otimes ^j \theta_2$. Hence, in $V^{\otimes k}$ we have the invariant $\theta_k = i^{-1}(\otimes^j \theta_2)$. The symmetric group $S_k$ acts on $V^{\otimes k}$ by permuting the factors. Then the following result of Weyl holds:

\begin{theorem}[\cite{W},\cite{H}]
 Let $G$ be one of $\OO(n)$ or $\Sp(2s)$. If $k$ is odd, there are no $G$-invariants in $V^{\otimes k}$. If $k$ is even, the translates of $\theta_k$ by $S_k$ span the $G$-invariants in $V^{\otimes k}$.
\end{theorem}

%%%%%%%%%%%%%%%%%%%%%%%%%%%%%%%%%%
\begin{comment}
We can formulate in similar terms the analogous statement for $G = \SL(n)$, which is also well-known. Let
\[
St_n(x_1, \dots, x_n) = \sum_{\sigma \in S_n} \mathrm{sign} (\sigma) x_{\sigma(1)}\otimes \cdots \otimes x_{\sigma(n)}
\]
denote the standard polynomial of degree $n$. It is well-known that $St_n$ spans the $\SL(n)$-invariants in $V^{\otimes n}$. Therefore, for any $j \geq 1$, in $\otimes^j V^{\otimes n}$ we have the invariant $\otimes ^j St_n$. Then, as above, for $k = jn$ we may choose an isomorphism
\[
i_k : V^{\otimes k} \cong \otimes^j V^{\otimes n}.
\]
Hence, in $V^{\otimes k}$ we have the invariant $R_k = i^{-1}(\otimes^j St_n)$. We have again the action of the symmetric group $S_k$ on $V^{\otimes k}$ by permuting the factors. Then, the following well-known theorem holds.

\begin{theorem}
 Let $G = \SL(n)$. If $k$ is not a multiple of $n$, there are no $G$-invariants in $V^{\otimes k}$. If $k = jn$ for some $j \geq 1$, the translates of $R_k$ by $S_k$ span the $G$-invariants in $V^{\otimes k}$.
\end{theorem}

\end{comment}
%%%%%%%%%%%%%%%%%%%%%%%%%%%%%%%%%%%%

An analogous result can be found in the literature for $G = \SL(n)$.

Propositions \ref{prop_HilbertSeries_SLSP}, \ref{prop_HilbertSeries_O}, and \ref{prop_HilbertSeries_SO} and Corollaries \ref{coro_beta} and \ref{coro_beta_UT} lead to the following further characterization of the $G$-invariants in $T(V)$. We include also the cases $G =\SL(n)$, $\SO(n)$ and $\mathrm{UT}(n)$ in our considerations.

\begin{theorem} \label{thm_invariants_SP}
Let $G = \SL(n)$ or $G = \Sp(2s)$ (for $n=2s$). For any $p \geq 1$, if $f$ is a $G$-invariant in $T(V)$ of degree greater than $n(p-1)$ then $f \in I_{p+1}$. % or $f$ can be expressed by invariants of smaller degree.
\end{theorem}

\begin{theorem} \label{thm_invariants_O}
Let $G = \OO(n)$.
\begin{itemize}
\item[(i)] If $f$ is a $G$-invariant in $T(V)$ of degree greater than $2$ then either $f \in I_3$ or $f$ is equal to a power of $\theta_2$.
\item[(ii)] For any $p \geq 3$ such that $p-1$ is an even number, if $f$ is a $G$-invariant in $T(V)$ of degree greater than $n(p-1)$ then either $f \in I_{p+1}$ or $f$ can be expressed by invariants of smaller degree.
\item[(iii)] For any $p \geq 3$ such that $p-1$ is an odd number, if $f$ is a $G$-invariant in $T(V)$ of degree greater than $n(p-2)$ then either $f \in I_{p+1}$ or $f$ can be expressed by invariants of smaller degree.
\end{itemize}
\end{theorem}

%Notice that Theorem \ref{thm_invariants_SP} holds also for $G = \SL(n)$. Using Proposition \ref{prop_HilbertSeries_SO} and Corollary \ref{coro_beta} one can state the following analogous result for $G = \SO(n)$.

\begin{theorem} \label{thm_invariants_SO}
Let $G = \SO(n)$.
\begin{itemize}
\item[(i)] If $f$ is a $G$-invariant in $T(V)$ of degree greater than $2$ then either $f \in I_2$ or $f$ can be expressed by invariants of smaller degree.
\item[(ii)] For any $p \geq 2$, if $f$ is a $G$-invariant in $T(V)$ of degree greater than $n(p-1)$ then either $f \in I_{p+1}$ or $f$ can be expressed by invariants of smaller degree.
\end{itemize}
\end{theorem}

\begin{theorem} \label{thm_invariants_UT}
Let $G = \mathrm{UT}(n)$.
\begin{itemize}
\item[(i)] If $f$ is a $G$-invariant in $T(V)$ of degree greater than $1$ then either $f \in I_2$ or $f$ can be expressed by invariants of smaller degree.
\item[(ii)] For any $p \geq 2$, if $f$ is a $G$-invariant in $T(V)$ of degree greater than $n(p-1)$ then either $f \in I_{p+1}$ or $f$ can be expressed by invariants of smaller degree.
\end{itemize}
\end{theorem}

\section*{Acknowledgements} I am very grateful to Vesselin Drensky for the productive discussions that we had during my work on this text.


\begin{thebibliography} {99}

\bibitem{BJ}
A. Bapat, D. Jordan,
{\it Lower central series of free algebras in symmetric tensor categories},
J. Algebra, {\bf 373} (2013), 299-311.

\bibitem{BBDGK}
F. Benanti, S. Boumova, V. Drensky, G. K.  Genov, P. Koev,
{\it Computing with rational symmetric functions and applications to invariant theory and PI-algebras},
Serdica Math. J. {\bf 38} (2012), Nos 1-3, 137-188.

%\bibitem{D}
%V. Drensky, {\it On lattices of varieties of associative algebras} (in Russian),
%Serdica Math. J. {\bf 8} (1982), No 1, 20-31.

\bibitem{Da}
R. Dangovski,
{\it On the maximal containments of lower central series ideals},
preprint, arXiv:1509.08030v2.

\bibitem{DeK1}
G. Deryabina, A. Krasilnikov,
{\it Products of commutators in a Lie nilpotent associative algebra},
Journal of Algebra {\bf 469} (2017), 84-95.

\bibitem{DeK2}
G. Deryabina, A. Krasilnikov,
{\it On some products of commutators in an associative ring},
International Journal of Algebra and Computation, 29:2 (2019), 333-341.

\bibitem{DD}
M. Domokos, V. Drensky,
{\it A Hilbert-Nagata Theorem in Noncommutative Invariant Theory},
Transactions of the AMS {\bf 350} (1998), No. 7, 2797-2811.

\bibitem{DD3}
M. Domokos, V. Drensky, 
{\it Rationality of Hilbert series in noncommutative invariant theory},
International J. Algebra and Computation {\bf 27} (2017), No. 7, 831–848.

\bibitem{DD2}
M. Domokos, V. Drensky,
{\it Constructive Noncommutative Invariant Theory},
Transformation Groups {\bf 26} (2021), No. 1, 215-228.


\bibitem{D}
V. Drensky,
{\it Codimensions of T-ideals and Hilbert Series of Relatively Free Algebras},
Journal of Algebra {\bf 91} (1984), No. 1, 1-17.

\bibitem{DH}
V. Drensky, E. Hristova,
{\it Noncommutative invariant theory of symplectic and orthogonal groups},
Linear Algebra and its Applications {\bf 581} (2019), 198-213.

\bibitem{FH}
W. Fulton, J. Harris,
{\it Representation Theory. A First Course.},
Graduate Texts in Mathematics, {\bf 129}, Readings in Mathematics, Springer-Verlag, New York, 1991.

\bibitem{GL}
N. Gupta, F. Levin,
{\it On the Lie ideals of a ring},
J. Algebra, {\bf 81} (1983), No. 1, 225-231.

\bibitem{H}
R. Howe,
{\it Remarks on Classical Invariant Theory},
Transactions of the AMS, {\bf 313} (1989), No. 2, 539-570.

\bibitem{J}
S. A. Jennings,
{\it On Rings Whose Associated Lie Rings are Nilpotent},
Bull. Amer. Math. Soc. {\bf 53} (1947), 593-597.

\bibitem{L}
V. N. Latyshev,
{\it On the finiteness of the number of generators of a $T$-ideal with an element
$[x_1, x_2, x_3, x_4]$},
Sibirskii Matematicheskii Zhurnal, {\bf 6} (1965), 1432-1434 (in Russian).

\bibitem{P}
S. V. Pchelintsev,
{\it Relatively free associative Lie nilpotent algebras of rank $3$},
Siberian Electronic Mathematical Reports {\bf 16} (2019) 1937-1946 (in Russian).

\bibitem{SV}
A. N. Stoyanova-Venkova,
{\it The Lattice of the Variety of Associative Algebras Defined by a Commutator of Length Five},
C. R. Acad. Bulg. Sci. {\bf 34} No. 4 (1981), 465-467.

\bibitem{V}
I. B. Volichenko,
{\it The T-ideal generated by the element $[X_1,X_2, x_3, x_4]$},
Preprint No. 22 (54), 1978, AN BSSR, Institute of Math. (in Russian)

\bibitem{W}
H. Weyl,
{\it The classical groups},
Princeton Univ. Press, Princeton, N.J., 1946.

\end{thebibliography}
\end{document}